\documentclass[12pt,a4paper]{article}

\usepackage[margin=20mm]{geometry}
\usepackage[T1]{fontenc}
\usepackage[latin2]{inputenc}
\usepackage{ae}
\usepackage{graphicx}
\usepackage[american]{babel}
\usepackage{amsfonts}

\usepackage{amsthm}
\usepackage{amsmath}
\usepackage{amssymb}
\usepackage{wasysym}
\usepackage{paralist}
\usepackage{multirow}
\usepackage{cite}
\usepackage{pifont}
\usepackage{imakeidx}
\makeindex
\newtheorem{Theorem}{Theorem}
\newtheorem{Lemma}[Theorem]{Lemma}

\newtheorem{Question}[Theorem]{Question}

\renewcommand{\d}{\ {\rm d}}
\newcommand{\E}{{\sf E}}
\newcommand{\G}{\mathbb G}
\newcommand{\I}{\mu}
\newcommand{\N}{\mathbb N}
\renewcommand{\P}{{\sf P}}
\renewcommand{\:}{\colon}

\usepackage{color}
\definecolor{red}{rgb}{0.8,0,0}
\definecolor{green}{rgb}{0,0.6,0}

\begin{document}

\title{On the graph limit question of Vera T. Sós}
\author{Endre Cs\'oka\footnote{Institute of Mathematics and DIMAP, University of Warwick, Coventry, United Kingdom.
Supported by ERC grants 306493 and 648017. 
}}
\date{}

\maketitle

\begin{abstract}
In the dense graph limit theory, the topology of the set of graphs is defined by the distribution of the subgraphs spanned by finite number of random vertices. Vera T. Sós proposed a question that if we consider only the number of edges in the spanned subgraphs, then whether it provides an equivalent definition. We show that the answer is positive on quasirandom graphs, and we prove a generalization of the statement.
\end{abstract}

\section{Introduction}

Graphon is a symmetric measurable function $W: [0, 1]^2 \rightarrow [0, 1]$, introduced in the context of graph theory by Lovász and Szegedy \cite{LoSz, lovasz2012large}. A sampling $\G(n, W)$ is a distribution of $n$-vertex graphs constructed as follows. We choose $x_1, x_2, ..., x_n \in [0,1]^n$ uniformly at random, and we connect the $i$th and $j$th vertices with probability $W(x_i, x_j)$. Therefore, the probability of each $n$-vertex graph $F = ([n], E(F))$ in the sampling $\G(n, W)$ is
\begin{equation*}
\G(n, W)(F) = \int\limits_{[0, 1]^n} \prod_{\{i, j\} \in E(F)} W(x_i, x_j) \prod_{\{i, j\} \in \overline{E(F)}} \big( 1 - W(x_i, x_j) \big) \ \d x.
\end{equation*}

A basic theorem about graphons tells that the sequence of samplings $\G(n, W)$ determine the graphon $W$ up to weak isomorphism \cite{lovasz2012large}. Vera T. Sós asked the question whether the distribution of the number of edges instead of the spanned subgraph already has this property \cite{Sos}. Formally, let
\begin{equation*}
\N(n, W)(k) = \sum_{F\: |E(F)| = k} \G(n, W)(F).
\end{equation*}

\begin{Question}
Does the sequence $\big(\N(n, W)\big)_{n \in \N}$ determine the sequence $\big(\G(n, W)\big)_{n \in \N}$?
\\ In other words, does the sequence $\N(n, W)$ determine the graphon $W$ up to weak isomorphism?
\end{Question}

Motivated by this question, Svante Janson (personal communication) asked the following question.

\begin{Question}
Does $\N(4, W) = \N(4, \frac{1}{2})$ imply $W \equiv \frac{1}{2}$, i.e. $W$ is constant $\frac{1}{2}$ almost everywhere?
\end{Question}

First, Jacob Fox proved the positive answer for all $p \in [0, 1]$ not only for $p = \frac{1}{2}$ in an unpublished result. At the same time, Noga Alon also proved for $p = \frac{1}{2}$. Later but independently, Jakub Sliacan proved it by flag algebra. In our paper we show a simple combinatorial proof of a stronger version of the statement. Namely, we show that 4 can be exchanged to any larger number, moreover, the sampling could be exchanged to any graph-sampling which includes a 4-cycle $C_4$.

Let subgraph mean edge-subgraph, namely we keep the vertex set, but we take a subset of its edges.
For a finite graph $S = ([n], E(S))$, we define the \emph{sampling} of a graphon $W$ by a graph $S$ to be the following distribution $\G(S, W)$ of the subgraphs of $S$.
\begin{equation*}
\G(S, W)(F) = \int\limits_{[0, 1]^n} \prod_{\{i, j\} \in E(F)} W(x_i, x_j) \prod_{\{i, j\} \in E(S) \setminus E(F)} \big( 1 - W(x_i, x_j) \big) \ \d x.
\end{equation*}
Analogously,
\begin{equation*}
\N(S, W)(k) = \sum_{F\: |E(F)| = k} \G(S, W)(F).
\end{equation*}

Notice that $\G(n, W) = \G(K_n, W)$ and $\N(n, W) = \N(K_n, W)$ (where $K_n$ is the complete graph on $n$ vertices).

\begin{Theorem} \label{main}
Let $G$ be a graph and $W$ a graphon. Assume that $G$ contains a $C_4$ and $\N(G, W) = \N(G, p)$. Then $W \equiv p$.
\end{Theorem}

\section{Proof of Theorem~\ref{main}}

Let $S_k$, $P_k$ and $C_k$ denote the star, the path and the cycle containing $k$ edges, respectively. Let $|E(G)| = m$, and we define homomorphism density of $F$ to $W$ as
\begin{equation} \label{tdef}
t(F, W) = \int\limits_{[0, 1]^{V(F)}} \prod_{\{i, j\} \in E(F)} W(x_i, x_j) \ \d x.
\end{equation}
We will use the simple fact that for vertex-disjoint union of graphs $F_i$,
\begin{equation} \label{prod}
t(F_1 \sqcup F_2 \sqcup ... \sqcup F_k, W) = \prod_{i = 1}^k t(F_i, W).
\end{equation}

\begin{Lemma} \label{tGk}
For an arbitrary positive integer $k\le m$, let $G_k$ denote the uniform random subgraph of $G$ with $k$ edges (namely, all ${m \choose k}$ subgraphs have equal probability). Then
\begin{equation} \label{tGkeq}
\E\big( t(G_k, W) \big) = p^k.
\end{equation}
\end{Lemma}

\begin{proof}
On the left hand side, there are two randomnesses: the choice of the subgraph $G_k$ and the sampling. But if we do the sampling first and we choose the subgraph after, then we get the right hand side. Formally,
\begin{equation} \label{expint}
\E\big( t(G_k, W) \big) \mathop{=}^{\eqref{tdef}} \E\Bigg( \int\limits_{[0, 1]^{V(G)}} \prod_{\{i, j\} \in E(G_k)} W(x_i, x_j)\ \d x \Bigg) = \int\limits_{[0, 1]^{V(G)}} \E\bigg( \prod_{\{i, j\} \in E(G_k)} W(x_i, x_j) \bigg) \d x.
\end{equation}
For a fixed $x \in [0, 1]^{V(G)}$, let $X: E(G) \rightarrow \{0, 1\} = \{false, true\}$ be independent events with probabilities $\P\big( X(\{i, j\}) \big) = W(x_i, x_j)$ for all $\{i, j\} \in E(G)$.
For a graph $F$, let us denote the number of occurring events by $\I(F) = \sum\limits_{(a, b) \in E(F)} X(a, b)$.
\begin{equation*}
\E_{G_k}\Big( \prod_{\{i, j\} \in E(G_k)} W(x_i, x_j) \Big) = \P_{G_k, X}\Big( \bigwedge_{\{i, j\} \in E(G_k)} X\big(\{i, j\}\big) \Big)
\end{equation*}
\begin{equation} \label{variables}
= \E_X \P_{G_k}\Big( \bigwedge_{\{i, j\} \in E(G_k)} X\big(\{i, j\}\big) \Big) = \E_X\bigg( \frac{{\I(G) \choose k}}{{m \choose k}} \bigg)
\end{equation}
Therefore,
\begin{equation} \label{eGk}
\E\big( t(G_k, W) \big) \mathop{=}^{\eqref{expint}\eqref{variables}} \int\limits_{[0, 1]^{V(G)}} \E_X\bigg( \frac{{\I(G) \choose k}}{{m \choose k}} \bigg) \d x = \E_{x, X} \bigg( \frac{{\I(G) \choose k}}{{m \choose k}} \bigg).
\end{equation}
The distibution of $\I(G)$ is $\N(G, W)$ by definition.
We can make the same calculation with $m$ independent edges $m \times P_1$ instead of $G$, which provides that
\begin{equation} \label{ekP1}
t(k \times P_1, W) = \E\Big( t\big((m \times P_1)_k, W\big) \Big) = \E_{x, X} \bigg( \frac{{\I(m \times P_1) \choose k}}{{m \choose k}} \bigg).
\end{equation}

$\N(G, W) = \N(G, p) =$ binomial distribution $B(m, p) = \N(m \times P_1, W)$, hence, $\I(G) = \I(m \times P_1)$. Therefore,
\begin{equation*}
\E\big( t(G_k, W) \big) \mathop{=}^{\eqref{eGk}} \E_{x, X} \bigg( \frac{{\I(G) \choose k}}{{m \choose k}} \bigg) =  \E_{x, X} \bigg( \frac{{\I(m \times P_1) \choose k}}{{m \choose k}} \bigg) \mathop{=}^{\eqref{ekP1}} t(k \times P_1, W) = t(P_1, W)^k = p^k. \qedhere
\end{equation*}
\end{proof}

This lemma for $k = 1$ immediately gives that
\begin{equation} \label{P1p}
E\big(t(P_1, W)\big) = E\big(t(G_1, W)\big) = p.
\end{equation}

\begin{Lemma}
\begin{equation} \label{tS2}
t(S_2, W) = p^2
\end{equation}
\end{Lemma}

\begin{proof}
We apply Lemma~\ref{tGk} with $k = 2$. The support of $G_2$ consists of two (isomorphism classes of) graphs: two independent edges $2 \times P_1$ and $S_2$. We 
Therefore, for some $\lambda > 0$,
\begin{equation*}
p^2 = \E \big( t(G_2, W) \big) = \lambda \cdot t(S_2, W) + (1 - \lambda) \cdot t(2 \times P_1, W) \mathop{=}^{\eqref{prod}} \lambda \cdot t(S_2, W) + (1 - \lambda) \cdot t(P_1, W)^2
\end{equation*}
\begin{equation*}
\mathop{=}^{\eqref{P1p}} \lambda \cdot t(S_2, W) + (1 - \lambda) \cdot p^2 = p^2 + \lambda \cdot \big( t(S_2, W) - p^2 \big),
\end{equation*}
which implies, $t(S_2, W) = p^2$.
\end{proof}

The \emph{degree} of a vertex $x \in [0, 1]$ of a graphon $W$ is defined as follows.
\begin{equation*}
\deg(x) = \int_0^1 W(x, y) \d y
\end{equation*}

Note that $W$ is measurable, therefore, $\deg(x)$ exists for almost all vertices $x \in [0, 1]$.

\begin{Lemma}
Almost all degrees of $W$ are $p$.
\end{Lemma}

\begin{proof}
\begin{equation*}
{\sf Var}\big(\deg(W)\big) = \E\big(\deg(W)^2\big) - \E\big(\deg(W)\big)^2 = t(S_2, W) - p^2 = 0. \qedhere
\end{equation*}
\end{proof}

\begin{Lemma} \label{extend}
Let $F$ be an arbitrary graph and $F'$ be its extension by one new vertex $v$ and a new edge $(v, w)$ connecting $v$ to an arbitrary old vertex. Then
\begin{equation} \label{extendeq}
t(F', W) = p \cdot t(F, W).
\end{equation}
\end{Lemma}

\begin{proof}
In short, whatever we sample by $F$, the probability that $(v, w)$ maps to an edge in $W$ is $p$, because all degrees are $p$. Formally,
\begin{equation*}
t(F', W) \mathop{=}^{\eqref{tdef}} \int\limits_{[0, 1]^{V(F')}} \prod_{\{i, j\} \in E(F')} W(x_i, x_j) \ \d x
= \int\limits_{[0, 1]^{V(F)}} \prod_{\{i, j\} \in E(F)} W(x_i, x_j) \int\limits_{[0, 1]} W(x_v, x_w) \d x_v \d x_{V(F)}
\end{equation*}
\begin{equation*}
= \int\limits_{[0, 1]^{V(F)}} \prod_{\{i, j\} \in E(F)} W(x_i, x_j) \cdot \deg(w) \d x_{V(F)} = p \cdot \int\limits_{[0, 1]^{V(F)}} \prod_{\{i, j\} \in E(F)} W(x_i, x_j) \d x \mathop{=}^{\eqref{tdef}} p \cdot t(F,W). \qedhere
\end{equation*}
\end{proof}

\begin{Lemma}
$t(P_3, W) = t(S_3, W) = p^3.$
\end{Lemma}

\begin{proof}
We apply Lemma~\ref{extend} for $F = P_2 = S_2$ and $F' = P_3$ or $F' = S_3$, namely,
\begin{equation*}
t(F', W) \mathop{=}^{\eqref{extendeq}} p \cdot t(S_2, W) \mathop{=}^{\eqref{tS2}} p \cdot p^2 = p^3. \qedhere
\end{equation*}
\end{proof}

\begin{Lemma}
If $G$ contains a triangle $K_3$, then $t(K_3, W) = p^3.$
\end{Lemma}

\begin{proof}
We apply Lemma~\ref{tGk} with $k = 3$. $G_3$ may contain only the following graphs: $3 \times P_1$, $P_2 \sqcup P_1$, $P_3$, $S_3$ and $K_3$. $G$ contains a $K_3$, therefore, $\P(G_3 = K_3) > 0$. We already know that $t(3 \times P_1, W) = t(P_2 \sqcup P_1, W) = t(P_3, W) = t(S_3, W) = p^3$, therefore, \eqref{tGkeq} implies that $t(K_3, W) = p^3$, as well.
\end{proof}

\begin{Lemma}
$t(C_4, W) = p^4$.
\end{Lemma}

\begin{proof}
We apply Lemma~\ref{tGk} with $k = 4$. Using the previous lemmas and applying Lemma~\ref{extend}, we see that for all subgraphs with 4 edges except $C_4$, the homomorphism densities are $p^4$. This implies $t(C_4, W) = p^4$.
\end{proof}

The theorem of Chung, Graham and Wilson \cite{chung1989quasi}, in the language of graphons \cite{lovasz2012large} shows that if $t(C_4, W) = t(P_1, W)^4 = p^4$, then $W \equiv p$. \qed

\section{Acknowledgement}

I want to say thank you to Oleg Pikhurko for suggesting me this question and for useful comments on the paper.

\bibliography{reference}

\end{document}